\documentclass[reqno, 12pt]{amsart}

\usepackage[algoruled,lined,noresetcount,norelsize]{algorithm2e}

\usepackage{graphicx}
\usepackage{amsmath,amssymb,amsthm}
\usepackage{mathrsfs}
\usepackage{mathabx}\changenotsign
\usepackage{dsfont}
\usepackage{xcolor}
\usepackage[backref]{hyperref}
\hypersetup{
    colorlinks,
    linkcolor={red!60!black},
    citecolor={green!60!black},
    urlcolor={blue!60!black}
}
\usepackage[T1]{fontenc}
\usepackage{lmodern}
\usepackage[babel]{microtype}
\usepackage[english]{babel}

\usepackage{geometry}
\numberwithin{equation}{section}
\numberwithin{figure}{section}

\newtheorem{theorem}{Theorem}

\newtheorem{corollary}[theorem]{Corollary}
\newtheorem{lemma}[theorem]{Lemma}

\newtheorem{problem}[theorem]{Problem}
\newtheorem{question}[theorem]{Question}
\newtheorem{remark}[theorem]{Remark}
\newtheorem{claim}[theorem]{Claim}
\newtheorem{definition}[theorem]{Definition}

\newcommand{\calF}{\mathcal{F}}

\title{Biased domination games}

\author{Ali Deniz Bagdas, Dennis Clemens, Fabian Hamann, Yannick Mogge}

\thanks{The research of the second and fourth author is supported 
by Deutsche Forschungsgemeinschaft (Project CL 903/1-1). }

\begin{document}

\maketitle

\begin{abstract}
We consider a biased version of Maker-Breaker domination games, which were recently introduced by Gledel, Ir{\v{s}}i{\v{c}}, and Klav{\v{z}}ar.
Two players, Dominator and Staller,
alternatingly claim vertices of a graph $G$
where Dominator is allowed to claim up to $b$
vertices in every round and she wins if and only if she occupies all vertices of a dominating set of $G$.
For this game, we prove a full characterization
of all trees on which Dominator has a winning strategy. For the number of rounds which Dominator needs to win, we give exact results when played on powers of paths or cycles, 
and for all trees we provide bounds which are optimal up to a constant factor not depending on $b$.
Furthermore, we discuss general minimum degree conditions and study how many vertices can still be dominated by Dominator even when Staller has a winning strategy.
\end{abstract}

\section{Introduction}

A dominating set $D$ in a graph $G=(V,E)$ is a set of vertices
such that each vertex in $V\setminus D$ has a neighbour in $D$.
The \textit{domination number} $\gamma(G)$ of a graph $G$, i.e.~the smallest size of a dominating set of $G$, belongs to the most prominent graph parameters in graph theory. While the decision problem whether $\gamma(G)\leq k$ holds is a classical example of an NP-complete decision problem~\cite{Karp72}, research in the last decades has brought many interesting results, including e.g.~concentration results for random graphs~\cite{bonato2015domination,glebov2015concentration,wieland2001domination}.

Note that a greedy algorithm, such as taking large-degree vertices one at a time, can lead to dominating sets that are very far apart from the optimum. Nevertheless, to model such a vertex-by-vertex approach, Bre{\v{s}}ar, Klav{\v{z}}ar, and Rall~\cite{brevsar2010domination}
introduced the \textit{domination game} played on a graph $G$ as follows. Two players, called Dominator and Staller, alternatingly choose vertices of $G$ with the restriction that any new vertex must enlarge the set of already dominated vertices, and the game ends once all vertices are dominated by the chosen vertices.
Dominator's goal is to minimize the number of moves,
while Staller's goal is to maximize this number.
The \textit{game domination number} then is the smallest number of moves until the game ends, under optimal play of both players,
and it is denoted with $\gamma_g(G)$ or $\gamma_g'(G)$,
depending on whether Dominator or Staller is the first player. 

Since the seminal paper~\cite{brevsar2010domination},
the above combinatorial game has been studied extensively,
including challenging open problems such as the 3/5-conjecture
for graphs without isolated vertices~\cite{kinnersley2013extremal} and the 1/2-conjecture for traceable graphs~\cite{james2019domination}.
Nevertheless, and partly due to the non-monotone behaviour of these games~\cite{brevsar2014domination,brevsar2013domination},
only few precise results are known for fixed graphs $G$, including
paths and cycles~\cite{kovsmrlj2017domination}, powers of cycles~\cite{bujtas2015domination}, and caterpillars~\cite{brevsar2013domination}.
Unfortunately, even for such simple graph classes, the study of the domination game tends to be complex;
and already for the class of trees no precise results are known for the parameters $\gamma_g(G)$ and $\gamma_g'(G)$ in general. For an overview on game domination numbers and variants we recommend the book~\cite{brevsar2021domination}.

\smallskip

Very recently, Gledel, Ir{\v{s}}i{\v{c}}, and Klav{\v{z}}ar~\cite{gledel2019maker} introduced the following Maker-Breaker variant of the above game.
Two players, again called Dominator and Staller,
alternatingly claim vertices of the given graph $G$.
Dominator wins if she manages to claim all the vertices of a dominating set of $G$, and Staller wins otherwise.
One main question then is to characterize all graphs for which Dominator has a winning strategy.
Provided that Dominator can win,
the other main question then is how fast she can do so. 
For this, let the \textit{Maker-Breaker domination number} 
$\gamma_{MB}(G)$ be the smallest number of rounds within which Dominator can always win the Maker-Breaker domination game on $G$, provided that she is the first player and both players play optimally, and
set $\gamma_{MB}(G)=\infty$ if Staller has a winning strategy. Moreover, define $\gamma'_{MB}(G)$ analogously for the case when Staller is the first player.

Although the rules of the above two types of games seem very similar, they can behave very differently.
For instance, it is known that $|\gamma_g(G)-\gamma'_g(G)|\leq 1$
for every graph $G$~\cite{brevsar2010domination,kinnersley2013extremal},
while $\gamma_{MB}(G)$ and $\gamma'_{MB}(G)$ can be arbitrarily
far apart~\cite{gledel2019maker}.

\begin{theorem}[Theorem~3.1~in~\cite{gledel2019maker}]
\label{thm:construction.Gledel}
For any integers $2\leq r\leq s\leq t$, there exists a graph $G$ such that $\gamma(G)=r$, $\gamma_{MB}(G)=s$
and $\gamma_{MB}'(G)=t$.
\end{theorem}

Moreover and amongst others, there exists a full characterization 
of all trees on which Dominators wins the Maker-Breaker domination game~\cite{duchene2020maker,gledel2019maker}, and precise results
on $\gamma_{MB}$ and $\gamma_{MB}'$ were proven
for all trees and cycles~\cite{gledel2019maker}.
To describe their result,
Gledel et al.~define a residual graph $R(G)$ which is obtained from $G$ by iteratively removing pendant paths of length $2$ until 
reaching a graph which does not have any further such pendant paths.

\begin{theorem}[Theorem~4.5~in~\cite{gledel2019maker}]\label{thm:gledel.trees}
If $T$ is a tree, then
$$
\gamma_{MB}(T) =
\begin{cases}
\frac{v(T)}{2}, & ~ \text{if $T$ has a perfect matching}\\
\frac{v(T)-1}{2}, & ~ \text{if $R(T)$ is a single vertex}\\
\frac{v(T)-k+1}{2}, & ~ \text{if $R(T)$ is a star with $k\geq 3$ edges}\\
\infty, & ~ \text{otherwise}
\end{cases}
$$
and
$$
\gamma_{MB}'(T) =
\begin{cases}
\frac{v(T)}{2}, & ~ \text{if $T$ has a perfect matching}\\
\infty, & ~ \text{otherwise.}
\end{cases}
$$
\end{theorem}

\begin{theorem}[Theorem~5.1~in~\cite{gledel2019maker}]
For every integer $n\geq 3$,
$$
\gamma_{MB}(C_n) =
\gamma_{MB}'(C_n) =
\left \lfloor \frac{n}{2} \right \rfloor .
$$
\end{theorem}

\smallskip

Motivated by the literature on Maker-Breaker games,
it seems natural to extend the above games by introducing a bias,
and with this paper we intend to prove first results in the case when a bias is given to Dominator. 
We note that we are not aware of any
papers considering domination games with a bias,
except from~\cite{bujtas2016disjoint}
which considers a variant called
\textit{Disjoint Domination Game}.

\smallskip

In general, a  biased \textit{$(p:q)$ Maker-Breaker game} on a hypergraph
$(X,\mathcal{F})$ is defined as follows:
Maker and Breaker alternatingly claim vertices of the \textit{board} $X$, with Maker always claiming up to $p$ vertices before Breaker
claims up to $q$ vertices. Maker wins if she occupies
a \textit{winning set} from $\mathcal{F}$ completely,
and Breaker wins if he can prevent such a situation throughout the game. For an overview on such games we recommend
the monograph~\cite{hefetz2014positional} and the survey~\cite{krivelevich2014positional}.

In these terms, we can define the \textit{$b$-biased Maker-Breaker domination game} on a graph $G$ as the $(b:1)$ Maker-Breaker game
where $X$ is the vertex set of $G$ and
$\mathcal{F}$ contains all dominating sets of $G$.
To be consistent with the previously introduced game, we call Maker the Dominator and we call Breaker the Staller.
Moreover, we define 
$\gamma_{MB}(G,b)$ and $\gamma'_{MB}(G,b)$
as the smallest number of rounds in which Dominator
wins the $b$-biased Maker-Breaker domination game on $G$
when Dominator or Staller start, respectively,
and where we set such a value to be $\infty$ if Staller
has a winning strategy.

As our first result, we extend the known results on paths and cycles by considering $b$-biased Maker-Breaker domination games on powers of paths and cycles. For this note that the $k$-th power $G^k$
of a graph $G$ is obtained by putting an edge between any two vertices of distance at most $k$.

\begin{theorem}\label{thm:powers.biased}
For all integers $b,k \leq n$ it holds that
$$
\gamma_{MB}(P_n^k,b) = 
\left\lceil 
\frac{n-1}{b(2k+1)-1}
\right\rceil 
~~ \text{and} ~~
\gamma_{MB}(C_n^k,b) = 
\left\lceil 
\frac{n-1}{b(2k+1)-1}
\right\rceil .
$$
\end{theorem}

Additionally, we give a full characterization of all trees
on which Dominator wins the $b$-biased Maker-Breaker domination game. Given a tree $T$ and a bias $b$, we say that $T$ is a $b$-good tree
if we can recursively delete vertices, which have exactly $b$ leaf neighbours, and also delete these leaf neighbours,
until we reach a forest where every vertex has at most $b-1$ leaf neighbours. Then the following holds.

\begin{theorem}\label{thm:tree.characterization}
Let $T$ be a tree with $v(T)\geq 2$, and let $b\geq 1$ be any integer. Then the following are equivalent:
\begin{enumerate}
\item[(i)] Dominator wins the $b$-biased Maker-Breaker domination game on $T$ when Staller is the first player.
\item[(ii)] $T$ is $b$-good.
\end{enumerate}
\end{theorem}

While the implication $(i)\Rightarrow (ii)$ is a simple exercise,
the proof of the opposite direction will be done by a more involved induction of a more general statement. Details will be given later in Section~\ref{sec:trees}. Moreover, the more general statement in this section will also allow for an analogous characterization for the case when Dominator is the first player.

Additionally, for the trees in which Dominator cannot win, we prove the following quantitative statement.

\begin{theorem}\label{thm:trees.fraction}
For every tree $T$ it holds that Dominator can dominate
at least $\left(1-\frac{1}{(b+1)^2}\right) v(T)$ vertices in the 
$b$-biased Maker-Breaker domination game on $T$. Moreover, the bound is sharp.
\end{theorem}

Looking at Theorem~\ref{thm:gledel.trees}, we see that,
when Staller starts, $\gamma'_{MB}(T)$
is of the same size for every tree $T$ on $n$ vertices for which $\gamma_{MB}'(T)<\infty$ holds. In the biased case, this changes drastically. Specifically, we prove the following.

\begin{theorem}\label{thm:trees.bounds}
Let $b\geq 2$ be any integer. Then the following holds.
\begin{enumerate}
\item[(a)] For every tree $T$ on $n$ vertices, if $\gamma_{MB}'(T,b)<\infty$ then
$$
\left \lceil \frac{n}{b(b+3)} \right \rceil \leq \gamma_{MB}'(T,b) \leq \left \lceil \frac{n}{b+1} \right\rceil.
$$
\item[(b)] 
Let $f(b):=(\lfloor \frac{b}{2} \rfloor + 1)(\lceil \frac{b}{2} \rceil + 1)$. Then
for every integer $t$ with 
$\lceil \frac{n}{f(b)} \rceil \leq t \leq \lceil \frac{n}{b+1} \rceil$ there exists a tree $T$ on $n$ vertices such that
$\gamma_{MB}'(T,b) = t$.
\end{enumerate}
\end{theorem}

Note that the lower bound in part (a) is optimal up to a constant factor 4, due to the trees guaranteed by part (b) and since
$4f(b)>b(b+3)$ for every $b$.
Moreover, a short discussion on $\gamma_{MB}(T,b)$
will be given later in this paper.

\smallskip

Finally, we also aim to provide some general statements on
$b$-biased Maker-Breaker domination games. The next two results in this regard give a tight minimum degree condition for Dominator to have a winning strategy and a general upper bound on $\gamma_{MB}(G,b)$.

\begin{theorem}\label{thm:min.degree.bound}
Let $n$ be a positive integer.
If $G$ is a graph with minimum degree $\delta(G) > \log_{b+1}(n) - 1$,
then Dominator wins the $b$-biased Maker-Breaker domination game on $G$.

Moreover, this bound on $\delta(G)$ is asymptotically best possible.
For infinitely many $n$, there is a graph $G$ on $n$ vertices
and with $\delta(G) > \log_{b+1}(n) - 2$ such that Staller wins
the $b$-biased Maker-Breaker domination game on $G$.
\end{theorem}

\begin{theorem}\label{thm:min.degree.bound2}
Let $G$ be a graph on $n\geq 2$ vertices  with 
minimum degree $\delta(G)\geq 10\ln(n)$,
and let $b\geq 1$ be an integer. Then
$$
\gamma_{MB}(G,b) \leq \frac{10n\ln(n)}{b\delta(G)}\, .
$$
\end{theorem}

Note that the above bound on $\gamma_{MB}(G,b)$ is optimal up to a logarithmic factor. This can be easily seen by considering a graph $G$ on $n$ vertices which is the disjoint union of copies of $K_{\delta+1}$, since then $\delta(G)=\delta$ and $\gamma_{MB}(G,b)=\lceil \frac{n}{b(\delta + 1)} \rceil$.

\bigskip

\textbf{Organization of the paper.} 
At first, in Section~\ref{sec:prelim},
we recall Beck's winning criterion for Maker-Breaker games.
Then, in Section~\ref{sec:powers}, we prove Theorem~\ref{thm:powers.biased} on powers of paths and cycles. In Section~\ref{sec:trees} we discuss Maker-Breaker domination games for trees,
including the proofs of Theorem~\ref{thm:tree.characterization},
Theorem~\ref{thm:trees.fraction}
and Theorem~\ref{thm:trees.bounds}.
Afterwards, in Section~\ref{sec:mindeg}
we prove Theorem~\ref{thm:min.degree.bound}
and Theorem~\ref{thm:min.degree.bound2}
providing minimum degree conditions for biased Maker-Breaker domination games.
We finish the paper with some concluding remarks in
Section~\ref{sec:concluding}.

\bigskip

\textbf{Notation.} 
Most of the notation used in this paper is standard 
and is chosen according to~\cite{west2001introduction}.
We set $[n]:=\{k\in\mathbb{N}:~ 1\leq k\leq n\}$ for
every positive integer $n$. 
Let $G$ be a graph. Then we write $V(G)$ and $E(G)$ for the vertex  set and the edge set of $G$, respectively, and we set 
$v(G) :=|V(G)|$ and $e(G):=|E(G)|$.
Additionally, if we have a family
$Q=(G_1,\ldots,G_k)$ of graphs we
let $V(Q)=\bigcup_{i\in [k]} V(G_i)$
and $v(Q):=|V(Q)|$.
Let $v,w$ be vertices in $G$.
If $\{v,w\}$ is an edge in $G$,  
we write $vw$ for short. 
The distance between $v$ and $w$
is the smallest number of edges
among all paths between $v$ and $w$ in $G$.
The neighbourhood of $v$ is
$N_G(v) : =\{w\in V(G): vw\in E(G)\}$,
and we call its size the degree of $v$.
We write
$N_G[v]:=N_G(v)\cup \{v\}$
for the closed neighbourhood of $v$.
Moreover, the minimum degree in $G$ is denoted $\delta(G)$.
Given any $A,B\subset V(G)$,
we set
$N_G(A):= \left(\bigcup_{v\in A} N_G(v)\right)\setminus A$, 
$N_G[A]:= N_G(A) \cup A$, and
$e_G(A,B):=|\{vw\in E(G):~ v\in A,w\in B\}|$.
Often, when the graph $G$ is clear from the context,
we omit the subscript $G$ in all definitions above.
 
Let $G$ and $H$ be graphs. 
We say that $H$ is a subgraph of $G$,
denoted $H\subset G$, if both $V(H)\subset V(G)$ and 
$E(H)\subset E(G)$ hold.
For a subset $A\subset V(G)$ we let $G-A$
be the subgraph of $G$ 
with vertex set $V(G)\setminus A$
and edge set $\{vw\in E(G):~ v,w\in V(G)\setminus A\}$.
Moreover, if $v\in V(G)$, then 
$G-v:=G-\{v\}$.

Let $F$ be a forest. Then we write $L(F)$ for the set of leaves, i.e.~all vertices of degree 1 in $F$. 
We write $P_n$ for the path on $n$ vertices, i.e.~with vertex set $V(P_n)=\{v_1,\ldots,v_n\}$
and edge set $E(P_n)=\{v_iv_{i+1}:~ i\in [n-1]\}$.
Similarly,
we write $C_n$ for the cycle on $n$ vertices, i.e.~with vertex set $V(C_n)=\{v_1,\ldots,v_n\}$
and edge set $E(C_n)=\{v_iv_{i+1}:~ i\in [n-1]\}\cup \{v_nv_1\}$.

We write Bin$(n,p)$ for the binomial random variable
with $n$ trials, each having success independently 
with probability $p$. We write $X\sim \text{Bin}(n,p)$ to   
denote that $X$ is distributed like Bin$(n,p)$,
and we say that an event, depending on $n$, holds asymptotically almost surely (a.a.s.) if its probability
tends to 1 as $n$ tends to infinity.
Moreover, for functions $f,g:\mathbb{N} \rightarrow \mathbb{R}$, we write $f(n)=o(g(n))$ if
$\lim_{n\rightarrow \infty} |f(n)/g(n)| = 0$.

Assume a Maker-Breaker domination game is in progress, then we call a vertex free if it is not claimed by any player so far.

\section{Preliminaries}\label{sec:prelim}

Later we will use the following version of a result due to Beck~\cite{beck1982remarks}, stated in the way
that (i) Breaker is the first player 
and additionally (ii) Breaker wants to bound the number of winning sets that Maker can occupy.
The proof of this version is almost the same as in~\cite{beck1982remarks}.
In~\cite{hefetz2014positional} (see Theorem 3.2.1 therein) and in~\cite{bednarska2000biased} (see comment before Lemma 5 therein) it can be found 
how Beck's original argument can be adapted to
allow for one of the two options (i) and (ii).
Moreover, it is straightforward to combine both options. 

\begin{theorem}[Beck Criterion]\label{thm:beck.criterion}
Let $q\in\mathbb{N}$, and let $(X,\calF)$ be a hypergraph. Then playing a $(1:q)$ Maker-Breaker game
on $(X,\calF)$, Breaker can ensure that
Maker occupies at most
$$\sum_{F\in \mathcal{F}} (1+q)^{-|F|}$$
winning sets, provided he is the first player.
\end{theorem}

Moreover, for some probabilistic arguments, we will use the following
Chernoff bounds (see e.g.~\cite{janson2011random}).

\begin{lemma}[Chernoff bounds]\label{lem:Chernoff}
	If $X \sim \text{Bin}(n,p)$, then
	\begin{itemize}
    		\item $\operatorname{Prob}(X<(1-\delta)np)
    				< \exp\left(-\frac{\delta^2np}{2}\right)$ 
    			for every $\delta>0$, and
    		\item $\operatorname{Prob}(X>(1+\delta)np)< \exp\left(-\frac{np}{3}\right)$ 
    			for every $\delta\geq 1$.
	\end{itemize}
\end{lemma}

\bigskip


\section{Powers of paths and cycles}\label{sec:powers}

In this section we prove Theorem~\ref{thm:powers.biased}.
Throughout the proof, we set $N:=b(2k+1)-1$.

\medskip

\textbf{Upper bound for $P_n^k$.}
Let $s\in \mathbb{N}$ be such that $2\leq s\leq N+1$ and $s \equiv n$ (mod $N$). Let $v_1,v_2,\ldots,v_n$ be the vertices
of $P_n$.
In the first round, Dominator claims the vertices 
$v_{k+1+(2k+1)i}$ with $0\leq i\leq b-1$, and in particular
dominates the first $s$ vertices $v_1,\ldots,v_s$.
Afterwards, consider the graph induced by
$v_{s+1},\ldots,v_n$, and note that its number of vertices is a multiple of $N$. From now on we can play only on this graph, assuming that Dominator is
the second player and no vertex has been claimed so far.
Split the vertex set into intervals of length $N$,
i.e. $I_i = \{v_{s+1+(i-1)N},\ldots,v_{s+iN}\}$ with $i\in \frac{n-s}{N}$.
In each of the next rounds $r=1,\ldots,\frac{n-s}{N}$, Dominator always chooses a distinct interval
and dominates it completely. In case that Staller touched
an interval for the first time in the previous round while this interval is not yet dominated by Dominator, Dominator chooses the same interval.

So, consider any such interval $I_i$ of length $N$, and assume that Staller claims
(at most) one vertex $v$ in it. It is easy to see that
the interval can be split into $b$ subintervals such that
all but one subinterval have length $2k+1$,
and the vertex $v$ is either in the unique subinterval of length $2k$,
or in one of the other subintervals but not in its middle.
Having such a partition, Dominator can always find a free element in each
of the subintervals with which the whole subinterval can be dominated.
Hence, by claiming only $b$ vertices, Dominator can dominate the whole interval $I_i$.

The number of rounds played is
$1+\frac{n-s}{N} = \lceil \frac{n-1}{N} \rceil$.
Hence, $\gamma_{MB}(P_n^k,b)\leq \lceil\frac{n-1}{b(2k+1)-1} \rceil$. 

\medskip

\textbf{Lower bound for $P_n^k$.}  
For the desired lower bound, we will apply an inductive argument,
but for a slightly stronger statement.
Let $\mathcal{Q}=(Q_1,\ldots,Q_{\ell})$ and
$\mathcal{R}=(R_1,\ldots,R_{\ell})$ be families of 
vertex-disjoint paths. We write $\mathcal{Q} \subset \mathcal{R}$ 
if $Q_i\subset R_i$ for every $i\in [\ell]$.
In this case we define the game 
$D_{MB}(\mathcal{Q},\mathcal{R},k,b)$ 
as follows: Maker (first player) and Breaker claim vertices of $\mathcal{R}$ according to bias $(b:1)$, and Maker wins if and only if she occupies a vertex set $F\subset V(\mathcal{R})$ such that each vertex of $V(\mathcal{Q})$ has distance at most $k$ to the set $F$. That is, the family of winning sets is
$$
\mathcal{F}:= \left\{F\subset V(\mathcal{R}):~ (\forall v\in V(\mathcal{Q})~ \exists w\in F:~ \operatorname{dist}(v,w)\leq k) \right\}.
$$
Moreover, we denote with  $\gamma_{MB}(\mathcal{Q},\mathcal{R},k,b)$ the smallest number of rounds in which Maker can win the game, and we change the subscripts to $BM$ if Breaker is the first player.

Inductively we will prove the following more general statement.

\begin{claim}\label{clm:paths.induction}
For all families $\mathcal{Q}=(Q_1,\ldots,Q_{\ell})$ and
$\mathcal{R}=(R_1,\ldots,R_{\ell})$ of 
vertex-disjoint paths such that $\mathcal{Q} \subset \mathcal{R}$, 
and for all integers $k,b$ we have
$$
\gamma_{MB}(\mathcal{Q},\mathcal{R},k,b)
\geq
\left\lceil \frac{v(\mathcal{Q})-1}{b(2k+1)-1} \right\rceil
~~ \text{ and } ~~
\gamma_{BM}(\mathcal{Q},\mathcal{R},k,b)
\geq
\left\lceil \frac{v(\mathcal{Q})}{b(2k+1)-1} \right\rceil .
$$
\end{claim}

Note that by setting $\mathcal{Q}=\mathcal{R}=\{P_n\}$ in the above claim, the desired bounds
for the Maker-Breaker domination game on $P_n^k$ follow.

\begin{proof}
We apply induction on $V(\mathcal{Q})$
and note that the claim is true if $v(\mathcal{Q})\leq N$,
since $\gamma_{MB}(\mathcal{Q},\mathcal{R},k,b)\geq 1$ and $\gamma_{BM}(\mathcal{Q},\mathcal{R},k,b)\geq 1$ hold trivially.
For the induction step, we separately consider the cases
that Maker starts the game or Breaker is the first player.

\underline{Case 1: Maker starts.}
In the first round Maker claims up to $b$ vertices.
Afterwards, we can colour these vertices red,
and use colour blue for all other vertices which have distance at most $k$ to a red vertex.
Then we let $\mathcal{Q}'=(Q_1',\ldots,Q_{\ell'}')$ be a family consisting of all maximal paths induced by the yet uncoloured vertices, and note that
$v(\mathcal{Q}')\geq v(\mathcal{Q})-b(2k+1) = v(\mathcal{Q}) - N - 1$.
Moreover, we define the family 
$\mathcal{R}'=(R_1',\ldots,R_{\ell'}')$ as follows: 
we extend each path $Q_i'$ to a path $R_i'$ by adding back all the blue vertices from $V(\mathcal{R})$ that are connected to $Q_i$ even after all red vertices are deleted. Note that the blue vertices in a path $R_i'$ have distance larger than $k$ to every path $Q_j'$ with $j\neq i$.
Hence, Breaker can consider to play the game $D_{BM}(\mathcal{Q}',\mathcal{R}',k,b)$ and play a strategy which
prevents that Maker wins within less than 
$\gamma_{BM}(\mathcal{Q}',\mathcal{R}',k,b)$ further rounds.
By induction, we thus obtain
$$
\gamma_{MB}(\mathcal{Q},\mathcal{R},k,b) 
\geq 
1 + \gamma_{BM}(\mathcal{Q}',\mathcal{R}',k,b)
\geq
1 + \left\lceil \frac{v(\mathcal{Q}')}{N} \right\rceil
\geq 
\left\lceil \frac{v(\mathcal{Q}')+N}{N} \right\rceil
\geq 
\left\lceil \frac{v(\mathcal{Q}) - 1}{b(2k+1)-1} \right\rceil .
$$

\underline{Case 2: Breaker starts.}
If all paths in $\mathcal{Q}$ have at most $2k$ vertices,
then the number of paths in $\mathcal{Q}$ is at least $\frac{v(\mathcal{Q})}{2k}$ and thus
$\gamma_{MB}(\mathcal{Q},\mathcal{R},k,b) 
\geq \frac{v(\mathcal{Q})}{2bk}
\geq \frac{v(\mathcal{Q})}{N}$.
Hence, from now on we can assume that $\mathcal{Q}$
contains a path, say $Q_1$, of length at least $2k+1$.
Label the vertices of $Q_1$ with $v_1,v_2,\ldots,v_{s}$
such that $s\geq 2k+1$,
and label the additional vertices of $R_1$ with
$\ldots,v_{-1},v_0$ (vertices attached to $v_1$)
and $v_{s+1},v_{s+2},\ldots$ (vertices attached to $v_s$).

In the first round Breaker claims $v_{k+1}$.
Afterwards and as long as Maker does not claim
a vertex $v_i$ with $i\leq k$, Breaker claims the vertices
$v_k,v_{k-1},v_{k-2},\ldots$ in this order.
If Maker never claims a vertex $v_i$ with $-k+1\leq i\leq k$,
she looses since she will then not manage to occupy a vertex of distance at most $k$ to the vertex $v_1$.
So assume that for some $r$, round $r$ is the first round
in which Maker claims a vertex $v_i$ with $-k+1\leq i\leq k$.
At that moment, Breaker has already claimed all vertices
$v_{k-r+2},\ldots,v_{k+1}$, and hence $i\leq k-r+1$.
So, for $v_i$ there can be at most $2k-r+1$ vertices in $V(\mathcal{Q})$ of distance at most $k$ to $v_i$.
In particular, it follows that after Maker's $r^{\text{th}}$ move, 
since for every other Maker's vertex there can be at most $2k+1$ vertices of distance at most $k$,
there are in total at most
$$(2k+1-r) + (rb-1)(2k+1) = r(b(2k+1)-1)=rN$$
vertices in $\mathcal{Q}$ which have distance at most $k$ to a Maker's vertex.

Analogously to Case 1, but now after round $r$,
we define red vertices (claimed by Maker) and blue vertices (distance at most $k$ to a red vertex), and using these colours,
we define $\mathcal{Q}'$ and $\mathcal{R}'$ as before.
By the previous estimate, we have
$v(\mathcal{Q}')\geq v(\mathcal{Q})-rN$, and $r$ rounds were played already.
Now, Breaker is the next player and w.l.o.g.~he can assume that
all his previously claimed elements are free again
so that he can consider to play the game
$D_{BM}(\mathcal{Q}',\mathcal{R}',k,b)$ from now on.
Applying induction analogously to Case 1,
we can conclude
$$
\gamma_{MB}(\mathcal{Q},\mathcal{R},k,b) 
\geq 
r + \gamma_{BM}(\mathcal{Q}',\mathcal{R}',k,b)
\geq
r + \left\lceil \frac{v(\mathcal{Q}')}{N} \right\rceil
=
\left\lceil \frac{v(\mathcal{Q}')+rN}{N} \right\rceil
\geq 
\left\lceil \frac{v(\mathcal{Q})}{b(2k+1)-1} \right\rceil .
$$

\medskip

\textbf{Bounds for $C_n^k$.} 
When $n\leq b(2k+1)$, then Dominator can dominate all vertices
within the first round, and thus $\gamma_{MB}(C_n^k,b)=1$.
Moreover, for every $n$ we have
$P_n^k\subset C_n^k$, from which it immediately follows that
$\gamma_{MB}(C_n^k)\leq \gamma_{MB}(P_n^k) = \left\lceil \frac{n-1}{b(2k+1)-1} \right\rceil$.
Hence, it is enough to give a strategy for Staller that prevents
Maker from winning the $(b:1)$ Maker-Breaker domination game on $C_n^k$ within less than $\left\lceil \frac{n-1}{b(2k+1)-1} \right\rceil$ rounds, when $n>b(2k+1)$.

After Maker's first move, at most $b(2k+1)$ are dominated.
As before, colour all claimed vertices red, and colour all other dominated vertices blue. Looking only at the underlying path $P_n$, the uncoloured vertices induce a family of paths $\mathcal{Q}'$ with $v(\mathcal{Q}')\geq n-b(2k+1)$, 
and using the blue vertices attached to these paths, 
we obtain a family $\mathcal{R}'$ with $\mathcal{Q}'\subset \mathcal{R}'$. Now, the game can be reduced to $D_{BM}(\mathcal{Q}',\mathcal{R}',k,b)$, and we already know from Claim~\ref{clm:paths.induction} that 
$\gamma_{BM}(\mathcal{Q}',\mathcal{R}',k,b)\geq 
  \left\lceil \frac{n-b(2k+1)}{b(2k+1)-1} \right\rceil$.
Hence, Staller can play in such a way that Dominator needs at least $\left\lceil \frac{n-b(2k+1)}{b(2k+1)-1} \right\rceil$
further rounds for winning on $C_n^k$, leading to a total of at least $1+\left\lceil \frac{n-b(2k+1)}{b(2k+1)-1} \right\rceil = \left\lceil \frac{n-1}{b(2k+1)-1} \right\rceil $ rounds.
\end{proof}

\bigskip


\section{Trees}\label{sec:trees}

\subsection{Characterization}\label{sec:characterize}
As a first step we want to prove Theorem~\ref{thm:tree.characterization}
by induction.
In the induction step we then have to deal with forests
from which some set of vertices $A$
is already claimed by Dominator, and in particular $N[A]$ is dominated. We therefore extend the definition of $b$-goodness, as given immediately before Theorem~\ref{thm:tree.characterization}, in such a way that it considers such a set $A$. Specifically, we will use the following notions. Recall that $L(F)$ denotes the set of leaves of a given forest $F$.

\begin{definition}
Let $T$ be a tree, let $A\subset V(T)$, let $b\in\mathbb{N}$, and  let $\mathbf{v}=(v_1,v_2,\ldots,v_t)$ be a sequence
of distinct vertices in $V(T)\setminus A$. 
Then we define a forest $T_{\downarrow \mathbf{v},A}$
as follows:
Set $F_0:=T$ and  
$F_i := F_{i-1} - (\{v_i\}\cup (N(v_i)\cap L(F_{i-1})\setminus A))$
for every $i\in [t]$.
Finally, let $T_{\downarrow \mathbf{v},A} := F_t$.
Moreover, we define the following:
\begin{itemize}
\item We say $\mathbf{v}$ is $(A,b)$-admissible for $T$ if 
$|N(v_i)\cap L(T_{\downarrow (v_1,\ldots,v_{i-1}),A})\setminus A|=b$ for all $i\in[t]$.
\item For $u\notin A$, we say that the sequence 
$(\mathbf{v}|u)=(v_1,\ldots,v_t,u)$ 
is $(A,b)$-problematic for $T$ if 
$\mathbf{v}$ is $(A,b)$-admissible for $T$ and
$|N(u)\cap L(T_{\downarrow \mathbf{v},A})\setminus A|\geq b+1$.
\item We say that $T$ is $(A,b)$-good if there
is no $(A,b)$-problematic sequence for $T$.
\end{itemize}
\end{definition}

Note that $b$-goodness now is the same as $(\varnothing,b)$-goodness. Hence, Theorem~\ref{thm:tree.characterization}
follows from the following slightly
more general theorem by setting $A=\varnothing$.

\begin{theorem}\label{thm:tree.A-b-good}
Let $T$ be a tree with $v(T)\geq 2$ and $A\subset V(T)$, then the following equivalence holds:
\begin{align*}
\left.
\begin{array}{c}
\text{When Staller starts, Dominator has a strategy to dominate all}\\ 
\text{vertices of $V(T)\setminus N[A]$ in the $(b:1)$ game on $T$} \\
\text{by claiming vertices in $V(T)\setminus A$ only.}
\end{array} \right\}
~ \Leftrightarrow ~
\text{$T$ is $(A,b)$-good.}
\end{align*}
\end{theorem}

\begin{proof}[Proof of Theorem~\ref{thm:tree.A-b-good}] 

\textbf{Proof of "$\Rightarrow$":}
We do a proof by contraposition.
Assume that $T$ is not $(A,b)$-good. 
Then, by definition, there is an $(A,b)$-problematic sequence
$(\mathbf{v}|u) = (v_1,v_2,\ldots,v_t,u)$ for $T$.
Note that by definition $v_1,v_2,\ldots,v_t,u\notin A$.
Staller can play as follows:
In the first round, 
Staller claims $v_1$, and Dominator is forced
to claim all the $b$ vertices of 
$N(v_1)\cap L(T)\setminus A$,
as otherwise Staller could block one of these vertices
immediately within his next move.
(For this note that all these $b$ vertices
belong to $V(T)\setminus N[A]$, since their only neighbour is $v_1\notin A$.)
Afterwards note that no vertex of 
$V(T_{\downarrow (v_1),A})\setminus A$ is dominated yet. Then Staller claims $v_2$, and analogously Dominator is forced to claim the $b$ vertices of 
$N(v_2)\cap L(T_{\downarrow (v_1),A})\setminus A$.
Staller continues like this and claims $v_3,\ldots,v_t,u$
in the next rounds. After $v_i$ is claimed for $i\in [t]$,
Dominator is always forced to claim the $b$ vertices of
$N(v_i)\cap L(T_{\downarrow (v_1,\ldots,v_{i-1}),A})\setminus A$.
But then, immediately after $u$ is claimed by Staller,
Dominator cannot claim all the at least $b+1$ vertices
of 
$N(u)\cap L(T_{\downarrow (v_1,\ldots,v_i),A})\setminus A$. 
Then Staller can
claim one of these vertices, say $z$, in the next round, and Dominator can never dominate this particular vertex $z$.
By choice of $z$, it is clear that $z\notin A$.
Moreover, since $N(z)\subset N(\{v_1,v_2,\ldots,v_t,u\})$,
we also have $z\notin N(A)$. Hence, Dominator does not manage
to dominate all of $V(T)\setminus N[A]$.

\bigskip

\textbf{Proof of "$\Leftarrow$":}
We prove this direction by strong induction on $v(T)$.
The base case can be checked easily. So, for the induction step ($\leq k-1 \rightarrow k$),
consider any tree $T$ with $v(T)=k$
and any set $A\subset V(T)$ such that $T$ is $(A,b)$-good.
The goal is to show that
Dominator has a strategy to dominate all vertices of
$V(T)\setminus N[A]$ in the $(b:1)$ game on $T$
by only claiming vertices in $V(T)\setminus A$ and
when Staller is the first player.

Assume that in the first round, Staller claims
some vertex $w$.
Let $T_1,T_2,\ldots,T_{\ell}$ be the components
of $T-w$.
Moreover, for each $i\in [\ell]$, 
let $x_i$ be the unique neighbour of $w$ in $T_i$,
and let $A_i = V(T_i)\cap A$. 
We claim that the following statements hold.

\begin{claim}\label{claim:A-b-good.claim1}
Let $i\in [\ell]$.
Assume $T_i$ is not 
$(A_i,b)$-good.
Then there exists an
$(A,b)$-admissible sequence
$\mathbf{w}$ for $T$,
contained in $V(T_i)$,
such that
$x_i\in N(w) \cap L(T_{\downarrow \mathbf{w},A})\setminus A$.
\end{claim}

\begin{claim}\label{claim:A-b-good.claim2}
For every $i\in [t]$, 
$T_i$ is $(A_i\cup \{x_i\},b)$-good.
\end{claim}

Before we prove the above claims, let us explain first
how the induction can be completed.

By Claim~\ref{claim:A-b-good.claim1} it follows that
at most $b$ components $T_i$ are not
$(A_i,b)$-good. Indeed, for contradiction assume that this was wrong, and w.l.o.g.~say that $T_i$ 
is not $(A_i,b)$-good for every $i\in [b+1]$.
Then let $\mathbf{w}_i$ be
an $(A,b)$-admissible sequence for $T$
as promised by Claim~\ref{claim:A-b-good.claim1},
for each $i\in [b+1]$.
The concatenation 
$\mathbf{w} = \mathbf{w}_1 \circ \mathbf{w}_2 \circ \ldots \circ \mathbf{w}_{b+1}$ then is an $(A,b)$-admissible sequence for $T$
such that each of the vertices $x_i$ with $i\in [b+1]$
belongs to $N(w)\cap L(T_{\downarrow \mathbf{w},A})\setminus A$. Hence, $(\mathbf{w}|w)$ is an
$(A,b)$-problematic sequence for $T$,
in contradiction to $T$ being $(A,b)$-good.

In particular, by possibly reordering the components $T_i$, we can assume that there is some $r\in \{0,1,\ldots,b\}$
such that $T_i$ is $(A_i,b)$-good if and only if $i > r$.
Dominator's strategy for the first round then is to claim $x_1$ (even if $T_1$ is already $(A_1,b)$-good) and then claim the vertices
$x_2,x_3,\ldots,x_{r}$ (and skip his remaining elements if $r<b$). Note that this is allowed since
by Claim~\ref{claim:A-b-good.claim1} we have
$x_1,x_2,\ldots,x_{r}\notin A$.
 
For the next rounds, Dominator always plays on the same component $T_i$
on which Staller made his previous move.
By induction hypothesis and using Claim~\ref{claim:A-b-good.claim2}, for every $i\in [r]$,
Dominator can dominate in $T_i$ every vertex of $V(T_i)\setminus N[A_i\cup \{x_i\}]$
without claiming elements of $A_i\cup \{x_i\}$.
Together with the element $x_i$ from the first round,
Dominator thus manages to dominate
$V(T_i)\setminus N[A]$.
Similarly, for every $i>r$, since $T_i$ is $(A_i,b)$-good,
Dominator can dominate in $T_i$ every vertex of $V(T_i)\setminus N[A]$ without claiming vertices in $A_i$. As $w$ is dominated by $x_1$, Dominator hence dominates $V(T)\setminus N[A]$ completely.

\bigskip

Hence, it remains to prove Claim~\ref{claim:A-b-good.claim1}
and Claim~\ref{claim:A-b-good.claim2}.

\medskip

\textit{Proof of Claim~\ref{claim:A-b-good.claim1}:}
By assumption
there exists an $(A_i,b)$-problematic sequence 
$(\mathbf{v}|u)$ for $T_i$.

\medskip

Case 1: First, We consider the case that
there exists an $(A_i,b)$-problematic sequence 
$(\mathbf{v}|u)$ for $T_i$
such that
$x_i\in V(T_{\downarrow \mathbf{v},A_i})$.
Then $\mathbf{v}$ is an $(A,b)$-admissible sequence 
for $T$. However, since $T$ is $(A,b)$-good,
we know that $(\mathbf{v}|u)$ is not 
$(A,b)$-problematic in $T$.
Therefore, $u$ has
at most $b$ neighbours in $L(T_{\downarrow \mathbf{v},A})\setminus A$. But then, since
$u$ has at least $b+1$ neighbours in 
$L((T_i)_{\downarrow \mathbf{v},A_i})\setminus A_i$,
and since
$L((T_i)_{\downarrow \mathbf{v},A_i})\setminus A_i
\subset 
(L(T_{\downarrow \mathbf{v},A})\setminus A) \cup \{x_i\}$
holds, 
it immediately follows that
$N(u) \cap L((T_i)_{\downarrow \mathbf{v},A_i})\setminus A_i$
has size $b+1$ and contains $x_i$,
while
$|N(u)\cap L(T_{\downarrow \mathbf{v},A})\setminus A|=b$.
Now, setting $\mathbf{w}=(\mathbf{v}|u)$,
we see that $\mathbf{w}$ is an
$(A,b)$-admissible sequence for $T$,
contained in $V(T_i)$,
such that
$x_i\in N(w) \cap L(T_{\downarrow \mathbf{w},A})\setminus A$.

\medskip

Case 2: We next consider the case that 
$x_i\notin V(T_{\downarrow \mathbf{v},A_i})$
for every
$(A_i,b)$-problematic sequence 
$(\mathbf{v}|u)=(v_1,\ldots,v_t,u)$ for $T_i$.

We first observe that for every such
sequence, we have $x_i \notin \{v_1,\ldots,v_t,u\}$,
since otherwise $\mathbf{v}$ would be an
$(A,b)$-admissible sequence for $T$
with $|N(u)\cap L(T_{\downarrow \mathbf{v},A})\setminus A|\geq b+1$, contradicting that $T$ is $(A,b)$-good.
Hence, since $x_i\notin V(T_{\downarrow \mathbf{v},A_i})$, there must be some $j\in [t]$ such that 
$x_i\in N(v_j)\cap 
L(T_{\downarrow (v_1,\ldots,v_{j-1})A_i})\setminus A_i$.
Among all possible
$(A_i,b)$-problematic sequences for $T_i$,
we now choose $(\mathbf{v}|u)$
such that $k:=t-j$ is minimal.

Having $(\mathbf{v}|u)$ fixed, 
we define an auxiliary digraph $D_{\mathbf{v}}$ 
on $\{v_j,\ldots,v_t,u\}$ as follows: 
For simplicity write $v_{t+1}:=u$,
and for every $j\leq j_1<j_2 \leq t+1$ 
we put a directed edge $(v_{j_1},v_{j_2})$ in $D_{\mathbf{v}}$
if and only if
$$|N(v_{j_2})\cap L(T_{\downarrow (v_1,\ldots,v_{j_1}),A_i})\setminus A_i|
>
|N(v_{j_2})\cap L(T_{\downarrow (v_1,\ldots,v_{j_1-1}),A_i})\setminus A_i|.$$
This means, roughly speaking, that in the step when $v_{j_1}$ and its $b$ leaf-neighbours are deleted, the vertex $v_{j_2}$
receives a new leaf-neighbour.
Note that this can only happen, if there is a (unique) path
$(v_{j_1},v,v_{j_2})$ in 
$L(T_{\downarrow (v_1,\ldots,v_{j_1-1}),A_i})$
and the deletion of $v_{j_1}$ makes $v$ a leaf.

The underlying (undirected) graph of $D_{\mathbf{v}}$
is a forest, and by construction, in $D_{\mathbf{v}}$
every directed edge goes from a vertex of smaller index to a vertex of higher index. We moreover observe the following three simple properties:
\begin{enumerate}
\item[(a)] The vertex $u$ has an ingoing edge.

\noindent
Indeed, if this was not the case, then
$|N(u)\cap L(T_{\downarrow (v_1,\ldots,v_{j-1}),A_i})\setminus A_i|\geq b+1$. Let $\mathbf{v}'$ be obtained from $\mathbf{v}$ by deleting $v_j$.
Then $(\mathbf{v}'|u)$
would be an $(A_i,b)$-problematic sequence 
for $T_i$ such that
$x_i\in V(T_{\downarrow \mathbf{v}',A_i})$,
in contradiction to the assumption of Case 2.

\item[(b)]
No other vertex than $u$ has out-degree $0$.

\noindent
Indeed, assume instead that we could find
a vertex $v_{\ell}\neq u$ with out-degree $0$.
Let $\mathbf{v}'$ be obtained from $\mathbf{v}$ by deleting $v_{\ell}$.
Then $(\mathbf{v}'|u)$
would be an $(A_i,b)$-problematic sequence 
for $T_i$, contradicting the minimality of the parameter $k$.

\item[(c)]
No other vertex than $v_j$ has in-degree $0$.

\noindent
Indeed, assume instead that we could find
a vertex $v_{\ell}\neq v_j$ with in-degree $0$.
Then the sequence
$(v_1,\ldots,v_{j-1},v_{\ell},v_j,\ldots,v_{\ell -1},v_{\ell +1},v_t|u)$
would be an $(A_i,b)$-problematic sequence 
for $T_i$, contradicting the minimality of the parameter $k$.
\end{enumerate}

From (a)--(c) it follows that
$D_{\mathbf{v}}$ is the directed path
$(v_j,\ldots,v_{t+1})$. Since therefore
every vertex in this path except from $v_j$ 
has exactly one ingoing edge
and since $(\mathbf{v}|u)$ is
$(A_i,b)$-problematic for $T_i$, 
we conclude  that
\begin{align*}
& |N(v_{\ell})\cap L(T_{\downarrow (v_1,\ldots,v_{j-1}),A_i})\setminus A_i|=b-1 ~~
\text{for every }j<\ell\leq t\, ,\\
& |N(u)\cap L(T_{\downarrow (v_1,\ldots,v_{j-1}),A_i})\setminus A_i|=b\, ,
\end{align*}
and that in $T_{\downarrow (v_1,\ldots,v_{j-1}),A_i}$
for every $j<\ell\leq t+1$ there must be a (unique) vertex $w_{\ell}$ of degree 2
such that $(v_{\ell-1},w_{\ell},v_{\ell})$ forms a path.
But then setting
$\mathbf{w}:=(v_1,\ldots,v_{j-1},u,v_t,\ldots,v_{j+1})$, we have $(\mathbf{w} | v_j)$ is an $(A,b)$-problematic sequence for $T$, contained in $V(T_i)$, such that $x_i\in N(w) \cap L(T_{\downarrow \mathbf{w},A})\setminus A$.\hfill $\Box$.

\bigskip

\textit{Proof of Claim~\ref{claim:A-b-good.claim2}:}
Assume the statement is wrong.
Then we could find
an $(A_i\cup \{x_i\},b)$-problematic sequence 
$(\mathbf{v}|u)$ for $T_i$.
This sequence would then also be
$(A,b)$-problematic for $T$,
in contradiction to $T$ being $(A,b)$-good.
\hfill $\Box$

\bigskip

This completes the proof of Theorem~\ref{thm:tree.A-b-good}.
\end{proof}

While Theorem~\ref{thm:tree.characterization} is a direct consequence of the just proven theorem, we can also conclude the following variant when Dominator is the first player.

\begin{corollary}\label{cor:tree.A-b-good_D.starts}
Let $T$ be a tree. When Dominator starts, then 
Dominator wins the $b$-biased Maker-Breaker domination game on $T$
if and only if there exists $A\subset V(T)$ of size at most $b$ such that $T$ is $(A,b)$-good.
\end{corollary}

\begin{proof}
If there exists $A \subset V(T)$ of size $b$
such that $T$ is $(A,b)$-good,
then in his first move, Dominator can occupy $A$
and afterwards win according to 
Theorem~\ref{thm:tree.A-b-good}.

Otherwise, if no such $A$ exists, then no matter which 
set $A$ Dominator claims in the first round, 
$T$ is not $(A,b)$-good. Hence,
by Theorem~\ref{thm:tree.A-b-good} Staller has a strategy
to make sure that Dominator will not dominate everything
in $V(T)\setminus N[A]$.
\end{proof}

\smallskip

\subsection{Maker-Breaker domination number for trees}

For the game in which Dominator starts, note that trivially
$$1\leq \gamma_{MB}(T,b)\leq \left\lceil \frac{n}{b+1} \right\rceil$$
holds for every tree $T$ on $n$ vertices,
provided that $\gamma_{MB}(T,b)<\infty$.
At the end of this subsection, Remark~\ref{remark:D.starts} will explain why every integer in this range can be realized by some tree $T$.
By proving Theorem~\ref{thm:trees.bounds}, we will however prove a significantly larger lower bound for $\gamma_{MB}'(T,b)$, which is linear in $n$. For the proof, we will make use of the following statement which can be seen as a generalization 
of Theorem 4.4.(i) in~\cite{gledel2019maker}.

\begin{lemma}\label{lem:residue}
Let $T$ be a tree which is $b$-good for some $b\in\mathbb{N}$, and let 
$\mathbf{v}=(v_1,v_2,\ldots,v_t)$ be 
an $(\varnothing,b)$-admissible sequence for $T$.
Then
$$
\gamma_{MB}'(T,b) = \frac{v(T)-v(T_{\downarrow \mathbf{v}})}{b+1} + \gamma_{MB}'(T_{\downarrow \mathbf{v}},b).
$$
\end{lemma}

\begin{proof}
Let $\mathbf{v}=(v_1,v_2,\ldots,v_t)$ be 
an $(\varnothing,b)$-admissible sequence for $T$,
and note that $t(b+1)=v(T)-v(T_{\downarrow \mathbf{v}})$.

\textbf{"$\geq$":}
In the first $t$ rounds, Staller can claim the
vertices $v_1,v_2,\ldots,v_t$, always forcing
Dominator to claim $b$ leaf-neighbours (as in the proof of
Theorem~\ref{thm:tree.A-b-good}).
Afterwards, the game reduces to $T_{\downarrow \mathbf{v}}$ where no vertex is claimed yet and no vertex is dominated by Dominator so far.
Hence, Staller can continue the game so that it lasts
at least $\gamma_{MB}'(T_{\downarrow \mathbf{v}},b)$ further rounds,
leading to a total of at least
$t+\gamma_{MB}'(T_{\downarrow \mathbf{v}},b)$
rounds.

\smallskip

\textbf{"$\leq$":}
For each $i\in [t]$,
let $V_i = \{v_i\} \cup (N(v_i)\cap L(T_{\downarrow (v_1,\ldots,v_{i-1}),\varnothing}))$, and note
that each $V_i$ induces a star with $b$ edges.
Now, Dominator can play according to the following strategy: Whenever Staller claims a vertex in some $V_i$,
Dominator claims all remaining vertices in $V_i$, 
and this way dominates all of $V_i$.
Whenever Staller plays on $V(T_{\downarrow \mathbf{v}})=V\setminus (\bigcup_{i\in [t]} V_i)$ and Dominator does not dominate $V(T_{\downarrow \mathbf{v}})$ yet,
Dominator plays on $T_{\downarrow \mathbf{v}}$ with the strategy that ensures to win within $\gamma_{MB}'(T_{\downarrow \mathbf{v}})$ rounds. 
Whenever Staller plays on $V(T_{\downarrow \mathbf{v}})$ while 
Dominator already dominates $V(T_{\downarrow \mathbf{v}})$,
then Dominator claims
$b$ vertices of a set $V_i$ which is not yet dominated.
Following the strategy, Dominator wins within
$t+\gamma_{MB}'(T_{\downarrow \mathbf{v}},b)$ rounds.
\end{proof}

With the above lemma in hand, we can now prove Theorem~\ref{thm:trees.bounds}.

\smallskip

\begin{proof}[Proof of Theorem~\ref{thm:trees.bounds}]
\textbf{(a)} The upper bound is trivial as the game cannot last longer than $\lceil \frac{n}{b+1} \rceil$ rounds.
Hence, we focus on the lower bound from now on.
If $\gamma(T)\geq \frac{n}{b+3}$, then the lower bound in (a) holds trivially, since Dominator has less than
$\frac{n}{b+3}$ vertices as long as 
less than $\frac{n}{b(b+3)}$ rounds are played.
Thus, from now on we can assume that
$\gamma(T) < \frac{n}{b+3}$.
In particular, there is a set $A\subset V(T)$ 
of size $|A| < \frac{n}{b+3}$
such that $A\cup N_T(A)=V(T)$.
To each $w\in N_T(A)$ assign a unique vertex
$v\in A$ such that $vw\in E(T)$;
and for each vertex $v\in A$, denote with
$d^*(v)$ the number of vertices $w\in N_T(A)$
which are assigned to $v$ in this way.
Note that
$$
\sum_{v\in A} d^*(v) = |N_T(A)| = n - |A| > \frac{(b+2)n}{b+3}.
$$
Staller's strategy is as follows:
As long as less than $\frac{|A|}{b+1}$ rounds were played so far, 
Staller claims a free vertex in $A$ for which
$d^*(v)$ is maximal.
Let $A'\subset A$ be the set claimed this way. Then
$|A'|=\lceil \frac{|A|}{b+1} \rceil$ and
$$
\sum_{v\in A'} d^*(v) \geq 
\frac{1}{b+1} \sum_{v\in A} d^*(v)
> \frac{(b+2)n}{(b+1)(b+3)}.
$$
We claim that 
$T-A'$ has at least $\frac{n}{b+3}$ components. To see this,
let $c$ be the number of these components, and
consider the tree $T'$ obtained from $T$ by contracting the components of $T-A'$. Then $v(T')=|A'|+c$ and hence
$$
e(T')=v(T')-1=|A'|+c-1 .
$$
Moreover,
$$
e(T') \geq e_T(A',V(T)\setminus A')
\geq \sum_{v\in A'} d^*(v) \geq \frac{(b+2)n}{(b+1)(b+3)}
$$
and therefore
\begin{align*}
c \geq \frac{(b+2)n}{(b+1)(b+3)} - |A'| + 1 
& > \frac{(b+2)n}{(b+1)(b+3)} - \frac{|A|}{b+1} \\
& > \frac{(b+2)n}{(b+1)(b+3)} - \frac{n}{(b+1)(b+3)}
= \frac{n}{b+3}.
\end{align*}
Now note that Dominator needs at least one vertex in each component of $T-A'$ to win the $b$-biased Maker-Breaker domination game on $T$. As she can touch at most $b$ of these components in each round, she needs to play at least $\frac{n}{b(b+3)}$ rounds.

\bigskip

\textbf{(b)} 
Let $f(b):=(\lfloor \frac{b}{2} \rfloor + 1)(\lceil \frac{b}{2} \rceil + 1)$. The main part of the proof will be the following claim.

\begin{claim}\label{claim:trees.bound.reduction}
For all integers
$k\geq 1$ and $b\geq 2$, there exists a tree $T_{k,b}$ such that
$\gamma_{MB}'(T_{k,b}) = \lceil \frac{k}{f(b)} \rceil$.
\end{claim}

	\begin{figure}
		\centering
		\includegraphics[scale=0.7]{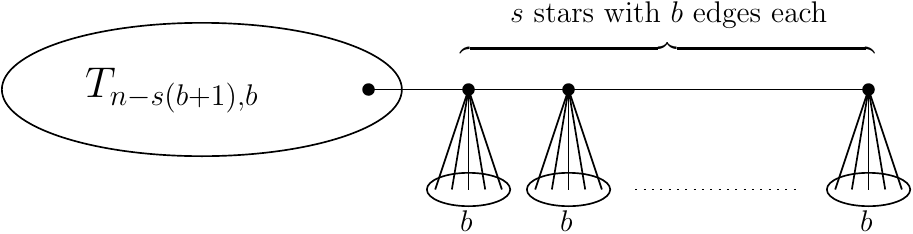}
		\caption{A tree $T(s)$}
		\label{fig:tree.construction}
	\end{figure}

Before we prove this claim, let us first explain how the statement (b) can be concluded.
Having $n$ and $b$ fixed, 
for any $0\leq s\leq \lfloor \frac{n}{b+1} \rfloor$
consider any tree $T(s)$ on $n$ vertices for which there is an $(\varnothing,b)$-admissible sequence $\mathbf{v}=(v_1,\ldots,v_s)$
with $T_{\downarrow \mathbf{v}} = T_{n-s(b+1),b}$;
an example is illustrated in Figure~\ref{fig:tree.construction}.
Then, using Lemma~\ref{lem:residue} 
and the claim above we obtain
$$
\gamma_{MB}'(T(S),b)=
s + \gamma_{MB}'(T_{n-s(b+1),b},b)=
s + \left\lceil \frac{n-s(b+1)}{f(b)} \right\rceil .
$$
Now for the expression $F(s):=s + \left\lceil \frac{n-s(b+1)}{f(b)} \right\rceil$ observe that
$F(0)=\lceil \frac{n}{f(b)}\rceil$,
and
$F( \lfloor \frac{n}{b+1} \rfloor ) = \lceil \frac{n}{b+1} \rceil$,
and
$F(j)\leq F(j+1)\leq F(j)+1$
for every $0\leq j\leq \lfloor \frac{n}{b+1} \rfloor$, since
$\frac{b+1}{f(b)}\leq 1$.
Hence, taking for $s$ all integer values
in the range $[0,\lfloor \frac{n}{b+1} \rfloor]$,
we see that $F(s)$ takes all integer values in the range $[ \lceil \frac{n}{f(b)} \rceil, \lceil \frac{n}{b+1} \rceil ]$, proving part (b) of the theorem.

\bigskip

It remains to prove Claim~\ref{claim:trees.bound.reduction}, which will be done in the following.

For short, let $x:=\lceil \frac{b}{2} \rceil$
and $q:= \lceil \frac{k}{x+1} \rceil$.
We set $r:=k-(q-1)(x+1)$ and note that
$1\leq r\leq x+1$.
Moreover, for later arguments note that
\begin{equation}\label{eq:trees.helpful.estimate}
\left( \left\lceil \frac{k}{f(b)} \right\rceil-1\right)
\cdot \left( b-x+1\right)
<
 \frac{k}{f(b)} 
\cdot \left( \left\lfloor \frac{b}{2} \right\rfloor + 1\right)
=
\frac{k}{x+1}.
\end{equation}

In order to construct the tree $T_{k,b}$,
we proceed as follows: we first take $q$ disjoint stars
$S_1,S_2,\ldots,S_{q}$ with centers $z_1,z_2,\ldots,z_q$
such that $e(S_i)=x$ for every $i\in [q-1]$
and $e(S_q)=r-1$, and then we we add the edges $z_iz_{i+1}$ with $i\in [q-1]$;
an illustration is given in Figure~\ref{fig:tree.construction2}.
Note that the resulting graph has $\sum_{i\in [q]} v(S_i)=k$
vertices.

	\begin{figure}
		\centering
		\includegraphics[scale=0.7,page=2]{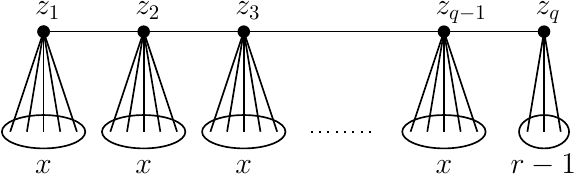}
		\caption{Tree $T_{k,b}$}
		\label{fig:tree.construction2}
	\end{figure}

We claim that $\gamma_{MB}'(T_{k,b}) = \lceil \frac{k}{f(b)} \rceil$. For this we distinguish the cases $r\neq 1$ and $r=1$.

\medskip

\underline{Case 1:} Assume that $r\neq 1$.

We first prove that 
$\gamma_{MB}'(T_{k,b}) \leq \lceil \frac{k}{f(b)} \rceil$.
To do so, let Dominator play as follows:
If Staller in his previous move claimed a vertex of a star $S_i$, $i\in [q]$, which does not contain any vertices claimed by Dominator, then Dominator first claims all remaining vertices
of $S_i$ (in total at most $x$ vertices),
and then she claims as many centers $z_i$ as possible
such that the corresponding stars $S_i$ do not contain a vertex of Staller, preferring $i<q$. 
This way, in each round Dominator makes sure to dominate all vertices of the just mentioned stars, and the number of these vertices is at least
$(b-x+1)(x+1) = f(b)$, except for maybe the last round. It follows that Dominator
claims a dominating set of $T_{k,b}$ within at most 
$\lceil \frac{k}{f(b)} \rceil$ rounds.

Next, we prove
$\gamma_{MB}'(T_{k,b}) \geq \lceil \frac{k}{f(b)} \rceil$.
To do so, consider the following strategy for Staller.
In the first $\lceil \frac{k}{f(b)} \rceil-1$ rounds,
Staller distinguishes three cases:
(1) If Staller can claim a vertex in such a way that he can complete the closed neighbourhood of a vertex, then he does so and wins the game. (2) If such a vertex does not exist, Staller claims a free star center $z_i$ with $i<q$. (3) Otherwise, if such a center also does not exist, Staller forfeits the game (we will see in a moment that this case cannot happen).
In round $\lceil \frac{k}{f(b)} \rceil$ Staller 
again plays according to (1) if possible, and otherwise
he plays arbitrarily.

\smallskip

In order to show that Staller succeeds
in making the game last at least $\lceil \frac{k}{f(b)} \rceil$
rounds, assume first that for some $j\leq \lceil \frac{k}{f(b)} \rceil-2$ Staller could follow his strategy and so far there was no option for following case (1). We now explain why then Staller can play his $(j+1)$-st move according to (2):
Let $z_{i_1},z_{i_2},\ldots,z_{i_{j}}$ be the star centers that Staller claimed in the first $j$ rounds.
Since (1) did not happen yet, Dominator must have claimed
all the remaining vertices of the corresponding stars
$S_{i_1},S_{i_2},\ldots,S_{i_{j}}$,
in total $jx$ vertices. As so far Dominator has claimed
$jb$ vertices in total, 
next to $S_{i_1},S_{i_2},\ldots,S_{i_{j}}$
she can have touched at most $j(b-x)$ further stars,
leading to a total of at most 
\begin{equation}\label{eq:trees.estimate1}
j(b-x+1) 
\leq 
\left( \left\lceil \frac{k}{f(b)} \right\rceil-2\right)
\cdot \left( b-x+1\right)
\stackrel{\eqref{eq:trees.helpful.estimate}}{<} \frac{k}{x+1} 
- (b-x+1)
\leq q - 1
\end{equation}
stars that may contain vertices of Dominator.
In particular, there must be a free center $z_i$ with $i<q$
so that Staller can follow (2). Thus, Staller can always follow
the proposed strategy without forfeiting the game.

Now, after at the end of round $\left \lceil \frac{k}{f(b)} \right \rceil - 1$, the same argument as above gives that Dominator can have touched at most
\begin{equation}\label{eq:trees.estimate2}
\left( \left\lceil \frac{k}{f(b)} \right\rceil-1\right)
\cdot \left( b-x+1\right)
\stackrel{\eqref{eq:trees.helpful.estimate}}{<} \frac{k}{x+1} 
\leq q
\end{equation}
stars so far, and hence there must be at least one star
where none of its leaves is yet dominated.
So, Dominator needs to play at least the round 
$\left\lceil \frac{k}{f(b)} \right\rceil$
before she can win.

\smallskip

\underline{Case 2:}  Assume that $r = 1$. The proof is essentially the same as in Case 1, but since $S_q$ is a single vertex,
Staller can force Dominator to claim $z_q$ by claiming the center $z_{q-1}$. In other words, we can instead from now on redefine our construction so that $z_q$ belongs to $S_{q-1}$, and hence
$e(S_{q-1})=x+1$ and $S_q$ does not exist.

For proving that 
$\gamma_{MB}'(T_{k,b}) \leq \lceil \frac{k}{f(b)} \rceil$,
let Dominator play as in Case 1.
In all rounds (except maybe the last round)
in which Staller does not claim $z_{q-1}$,
Dominator can again dominate $b-x+1$ stars completely,
and hence dominates $(b-x+1)(x+1)=f(b)$ vertices of these stars.
In a round (except maybe the last round)
in which Staller claims $z_{q-1}$,
Dominator must claim the $x+1$ leaves of $S_{q-1}$
and can then care about further $b-x-1$ stars by claiming their centers; the dominated vertices in these considered stars are then
$(b-x-1)(x+1)+(x+2) = f(b)-x$. 
Assume that Dominator would not win the game within the first 
$\lceil\frac{k}{f(b)}\rceil$ rounds, then with the same argument as above
she would still have dominated at least
$$
\left( \left\lceil\frac{k}{f(b)}\right\rceil - 1 \right)f(b) + 1\cdot (f(b)-x)
=
\left\lceil\frac{k}{f(b)}\right\rceil \cdot f(b) - x
$$
vertices by then. Now, note that by the assumption of the case we have
$k\equiv 1$ (mod $x+1$) and thus, since $f(b)$ is a multiple of $x+1$, we get $k\equiv t(x+1)+1$ (mod $f(b)$) for some integer $t$.
In particular, the number of dominated vertices can be further estimated from below by
$$
\left\lceil\frac{k}{f(b)}\right\rceil \cdot f(b) - x
\geq
\frac{k+x}{f(b)}\cdot f(b) - x = k,
$$
in contradiction to the assumption that Dominator has not won already.

\smallskip

For proving that 
$\gamma_{MB}'(T_{k,b}) \geq \lceil \frac{k}{f(b)} \rceil$,
let Staller also play as in Case 1,
except that in the first round he claims the center $z_{q-1}$
of the biggest star $S_{q-1}$.
Due to the additional leaf that Dominator needs to claim
in the first round when compared with Case 1,
Dominator looses one move (one vertex) for claiming a center of another star. Because of this, all the estimates in~\eqref{eq:trees.estimate1} and~\eqref{eq:trees.estimate2}
can be changed by adding the summand $-1$.
In particular, in~\eqref{eq:trees.estimate2} we then get that
Dominator has touched less than $q-1$ stars within the first
$\lceil \frac{k}{f(b)} \rceil - 1$ rounds, and hence at least one more round is needed for her to win the game.
\end{proof}

\begin{remark}\label{remark:D.starts}
For the case when Dominator starts the game, $\gamma_{MB}(T,b)$ can take every integer value $t$ between $1$ and $\frac{n}{b+1}$. For this consider the tree $T$ obtained as follows. Take $t$ disjoint stars $S_1,\ldots,S_{t}$ such that $e(S_i)=b$ for every $i<t$ and $e(S_t)=n-(b+1)t$, and add a path through the centers of these stars; then $\gamma_{MB}(T,b)=t$. We leave the details to the reader.
\end{remark}

\smallskip

\subsection{Number of dominated vertices}
Finally, we prove Theorem~\ref{thm:trees.fraction}.

\begin{proof}[Proof of Theorem~\ref{thm:trees.fraction}]
For the first part consider $\mathcal{F} = \{ N[v]:~ v\in V(G)\}$
and note that 
$$
\sum_{F\in \mathcal{F}} (1+b)^{-|F|}
\leq \frac{v(G)}{(b+1)^2} .
$$
By Theorem~\ref{thm:beck.criterion},
Dominator (playing as Breaker) can make sure that Staller claims at most
$\frac{v(G)}{(b+1)^2}$ sets of $\mathcal{F}$ completely.
This in turn implies the first part of the theorem.

For the sharpness, let $n\in\mathbb{N}$ and consider
a graph on $n$ vertices obtained as follows:
take $b+1$ stars each of size $\lfloor \frac{n}{b+1} \rceil$,
and add a path through the centers of these stars.
After the first move of Dominator, there is at least one star
which was not touched. Staller first claims the center,
and afterwards roughly a $\frac{1}{b+1}$-fraction of its leaves.
These leaves, the number of which is roughly $\frac{n}{(b+1)^2}$,
are not dominated by Dominator.
\end{proof}

\bigskip


\section{Minimum degree condition}\label{sec:mindeg}

In this section, we will make use of the following simple observation which can also be found in~\cite{beck1982remarks}.

\begin{lemma}\label{lemma:binary.tree}
Let $b\in\mathbb{N}$, 
and let $T$ be a perfect $(b+1)$-ary tree,
and let  $\calF = \calF(T)$ be the hypergraph 
on vertex set $V(T)$ where
the vertices of any root-leaf path of $T$ 
form a hyperedge of $\calF$. Then, Maker
wins the $(1:b)$ Maker-Breaker game on $(X,\calF)$,
provided she is the first player.
\end{lemma}

Having this result and Beck's Criterion in our hand, we can prove our first minimum degree condition (Theorem~\ref{thm:min.degree.bound}) for Maker-Breaker domination games.

\begin{proof}[Proof of Theorem~\ref{thm:min.degree.bound}]
For the first part of the theorem, let $G$ be any graph on $n$ vertices with $\delta(G) > \log_{b+1}(n) - 1$. 
Consider $\mathcal{F} = \{ N[v]:~ v\in V(G)\}$
and note that 
$$
\sum_{F\in \mathcal{F}} (1+b)^{-|F|}
\leq n (1+b)^{-(\delta(G)+1)} < n (1+b)^{-\log_{b+1}(n)} = 1\, .
$$
Hence, by Theorem~\ref{thm:beck.criterion}
it follows that Dominator (playing as Breaker) can claim an element in each set of $\mathcal{F}$. That is, Dominator claims a vertex in each closed neighbourhood, so that all of her vertices form a dominating set.

For the second part, let $T_k$ be a perfect $(b+1)$-ary tree 
with $k+1$ levels (the root is the first level), 
and let $G$ be the graph obtained from taking $(b+1)$ disjoints
copies of $T_k$ and adding additional edges between every vertex
and all its descendants.
Then $n:=v(G)= (b+1)\sum_{i=0}^k (b+1)^i < (b+1)^{k+2}$ and
$\delta(G) = k > \log_{b+1}(n)-2$.
We claim that Dominator loses the $(b:1)$ Maker-Breaker domination game on $G$.
For this, Staller plays as follows:
After Dominator's first move, where she claims $b$ vertices,
there will be at least one copy of $T_k$ untouched.
On this copy, Staller follows Maker's strategy guaranteed by Lemma~\ref{lemma:binary.tree}, so that Staller claims all the vertices of a root-leaf path. The leaf of this root-leaf path cannot be dominated by Dominator.
\end{proof}

Next we prove Theorem~\ref{thm:min.degree.bound2}
by combining a simple probabilistic argument 
and Beck's Criterion.

\begin{proof}[Proof of Theorem~\ref{thm:min.degree.bound2}]
At first observe that there exists a
subset $A\subset V(G)$
such that
\begin{equation}\label{eq:randomA}
|A|\leq \frac{10n\ln(n)}{\delta(G)}
~~ \text{and} ~~ 
|N(v)\cap A| > 2\ln(n) 
~~ \text{for every} 
~ v\in V(G).
\end{equation}
The existence can be proven with a standard probabilistic argument.
Indeed, let $A\subset V(G)$ be a random subset where
every vertex of $V(G)$ is contained in $A$ independently
with probability $p=\frac{9\ln(n)}{\delta(G)}$.
Then $\mathbb{E}(|A|)=np=\frac{9n\ln(n)}{\delta(G)}$ and, using a Chernoff inequality (Lemma~\ref{lem:Chernoff}), we have
$$
\operatorname{Prob}\left[ |A|> \frac{10n\ln(n)}{\delta(G)}\right]
\leq
\exp\left( - \frac{3n\ln(n)}{\delta(G)} \right)
\leq
\exp\left( - 3\ln(n) \right) = n^{-3}.
$$
Moreover, for every vertex $v$ we have
$\mathbb{E}(|N(v)\cap A|)=|N(v)|p \geq 9\ln(n)$ and
thus, using a Chernoff inequality (Lemma~\ref{lem:Chernoff}), we conclude
$$
\operatorname{Prob}\left[ |N(v)\cap A| \leq 2 \ln(n) \right]
\leq
\exp\left( - 2\ln(n) \right) = n^{-2}.
$$
Thus, by taking a union bound, we see that
\eqref{eq:randomA} fails with probability at most
$n \cdot n^{-2} + n^{-3} <1$.

\medskip

Now, fix such a set $A$ and consider
$\mathcal{F} := \{N(v)\cap A:~ v\in V(G)\}$.
Then
$$
\sum_{F\in \mathcal{F}} (1+b)^{-|F|}
\leq n (1+b)^{-2\ln(n)} < 1
$$
for every $b\geq 1$ and $n\geq 2$.
Thus, using Beck's Criterion (Theorem~\ref{thm:beck.criterion}), we see that by playing on $A$
only, Dominator (playing as Breaker) can get a vertex in each of the sets $A\cap N(v)$. As she has an element in each $N(v)$ then, she has claimed a dominating set.
Moreover, for playing on $A$ with bias $b$, Dominator only needs at most $|A|/b$ rounds, giving the desired bound.
\end{proof}

\bigskip


\section{Concluding remarks and open problems}\label{sec:concluding}

\textbf{Dominating trees.}
In part (a) of Theorem~\ref{thm:trees.bounds} we provide
general bounds on $\gamma_{MB}'(T,b)$
for any tree $T$, and with part (b) we  show that our general lower bound is optimal up to a multiplicative factor of at most $4$. Naturally this leads to the question of finding a tight lower bound and a characterization of extremal trees.

\begin{problem}
Given integers $b$ and $n$.
What is smallest possible value
$\gamma_{MB}'(T,b)$
among all trees $T$ on $n$ vertices?
\end{problem}

\begin{problem}
Given integers $b$ and $n$.
Characterize all trees $T$ on $n$ vertices,
which minimize $\gamma_{MB}'(T,b)$.
\end{problem}

While part (b) of Theorem~\ref{thm:trees.bounds} also ensures the existence of a tree $T$
such that $\gamma'_{MB}(T,b)=t$ whenever 
$\lceil \frac{n}{f(b)} \rceil \leq t \leq \lceil \frac{n}{b+1} \rceil$ with
$f(b):=(\lfloor \frac{b}{2} \rfloor + 1)(\lceil \frac{b}{2} \rceil + 1)$ holds, we do not a have a general approach
to determine $\gamma'_{MB}(T,b)$ or 
$\gamma_{MB}(T,b)$ for any given tree $T$.
In particular, we ask the following.

\begin{question}
Is there a polynomial time algorithm that finds
$\gamma'_{MB}(T,b)$ or $\gamma_{MB}(T,b)$ for any given tree $T$ and positive integer $b$?
\end{question}

\smallskip

\textbf{Dominating random graphs.}
In order to obtain a better understanding of biased
Maker-Breaker domination games, it also seems natural to
study such games when played on randomly generated graphs.

For instance, let $G\sim \mathcal{G}_2(n,r)$ be a random geometric graph, which is obtained by placing $n$ points uniformly at random into the square $[0,1]^2$ and by adding an edge between any two points if their Euclidean distance is at most $r$. Then, with a simple pairing strategy
together with an argument similar to the proof of
Theorem~2 in~\cite{bonato2015domination},
it can be shown that
$$\gamma_{MB}(G,b), \gamma'_{MB}(G,b) = \left( \frac{2\sqrt{3}}{9b} + o(1) \right) r^{-2}$$
holds a.a.s.
Moreover, if we instead consider the Erd\H{o}s-R\'enyi random graph model $G_{n,p}$ with constant $p$,
a simple randomized strategy for Dominator can be
used in order to show that a.a.s.
$$\gamma_{MB}(G,b), \gamma'_{MB}(G,b) = (1+o(1)) \frac{1}{b}\log_{1/(1-p)}(n)$$
holds. 
That is, in both of the above cases the number of elements that Dominator a.a.s.~needs to claim until winning the game on $G$ is asymptotically equal to the domination number $\gamma(G)$. We wonder whether even equality holds.

\begin{problem}
 Let an integer $b\geq 1$ be given, and let
 $G\sim \mathcal{G}_2(n,r)$ for some $r\in (0,1)$.
 Is it true that 
 $
 \gamma_{MB}(G,b) = \gamma'_{MB}(G,b) = \left\lceil\frac{\gamma(G)}{b} \right\rceil
 $
 holds a.a.s.?
\end{problem}

\begin{problem}
 Let an integer $b\geq 1$ be given, and
 let $G\sim G_{n,p}$ for some $p\in (0,1)$.
 Is it true that 
 $
 \gamma_{MB}(G,b) = \gamma'_{MB}(G,b) = \left\lceil\frac{\gamma(G)}{b} \right\rceil
 $
 holds a.a.s.?
\end{problem}

Note that if the answer to the last question is yes,
then $\gamma_{MB}(G,b)$ and $\gamma'_{MB}(G,b)$
would be concentrated on two values if $G\sim G_{n,p}$, see~\cite{glebov2015concentration,wieland2001domination}.
But even if the answer is no, it would be interesting to
analyse whether concentration results
similar to those in~\cite{glebov2015concentration,wieland2001domination}
can be proven.

\medskip

\textbf{Variants.} For all questions discussed in this paper, it also seems natural to consider variants of Maker-Breaker domination games, for example where Dominator's goal is to claim a total domination set or a $k$-dominating set for any positive integer $k$.

\smallskip

\section*{Acknowledgement}
Parts of this paper were obtained in the Bachelor thesis of the first author at the Hamburg University of Technology. The first author wants to thank Anusch Taraz and his research group
for the supervision.

\bibliographystyle{amsplain}
\bibliography{references}

\end{document}